\documentclass[11pt,reqno]{amsart}
\usepackage{amsmath,amssymb,amsthm}
\usepackage[hidelinks]{hyperref}
\newtheorem{thm}{Theorem}[section]

\theoremstyle{definition}

\newtheorem{lemma}[thm]{Lemma}
\newtheorem{corollary}[thm]{Corollary}
\newtheorem{theorem}[thm]{Theorem}
\newtheorem{conjecture}[thm]{Conjecture}

\numberwithin{equation}{section}

\def\N{\mathbb{N}}

\title{On the exponential Diophantine equation $(a^n+1)(b^n+1)=x^2$}

\author{Paulius Virbalas}
\address{Institute of Mathematics, Faculty of Mathematics and Informatics, Vilnius
University, Naugarduko 24, Vilnius LT-03225, Lithuania}
\email{paulius.virbalas@mif.vu.lt}

\begin{document}
\begin{abstract}
We study the Diophantine equation $(a^n+1)(b^n+1)=x^2$, which belongs to the family of equations originating from the work of Szalay in 2000. If $a>1$, it is shown that the equation of the title has only one solution in positive integers, when $a$ and $b$ are distinct powers of the same integer $t>1$. Also, a complete description of the solutions is obtained under the assumptions that $a$ and $b$ are coprime and $n$ is even. Several other special cases of the equation are considered, and two conjectures are proposed.
\end{abstract}

\maketitle

\section{\textbf{Introduction}}\label{intro}

Let $\N$ denote the set of positive integers. In 2000, Szalay \cite{sza00} initiated the study of the Diophantine equation
\begin{equation}\label{sza}
(a^n-1)(b^n-1)=x^2,\quad a,b,n,x \in \N,\quad  1<a<b,
\end{equation}
by investigating the cases 
\[(a,b) \in \{ (2,3), (2,5), (2,2^k)\},\]
where $k\geq 2$ is an integer. Since then, several authors have extended Szalay's results (see, for example, \cite{haj00, ishi16, sza10, le09,  li11, luc02, xia13, yuan12}). Cohn \cite{cohn02} conjectured in 2002 that equation \eqref{sza} has no solutions in positive integers for $n> 4$. However, a general approach to the full conjecture does not appear to be known at present. More tools are available when $n$ is even, since in this case \eqref{sza} can often be reduced to a system of two Pell equations
\[
a^n-dy_1^2=1
\quad\text{and}\quad
b^n-dy_2^2=1,
\]
where
\[
(a^n-1)(b^n-1)=dy_1^2 \cdot dy_2^2=x^2.
\]
If $n$ is even, then it is conjectured that any solution of \eqref{sza} satisfies either $(n,a,b)=(4,13,239)$ or $n=2$.
By exploiting properties of solutions to Pell equations, Keskin \cite{kes19}, as well as Noubissie, Togbé and Zhang \cite{tog20}, obtained several partial results towards the conjecture of Cohn.\par
The original equation \eqref{sza} also inspired the investigation of several related equations. For example, in 2024, Fujita and Le \cite{le24} initiated the study of
\[
(a^n-1)(b^m-1)=ax^2,\quad a,b,n,m,x \in \N,\quad 1<a<b,
\]
while in 2025, He and Liu \cite{lh26} considered the equation
\[
(2^n-1)(b^n-1)=x^m, \quad b,n,m,x \in \N,\quad b,n,m>1.
\]
For further variations, see \cite{kes22, lh25, tong21, wl20}. In this paper, we study the plus analogue of \eqref{sza}, namely
\begin{equation}\label{virb}
(a^n+1)(b^n+1)=x^2,\quad a,b,n,x \in \N,\quad a<b.
\end{equation}

The first result classifies all solutions to \eqref{virb}, when $a$ and $b$ are distinct powers (possibly trivial) of the same positive integer $t$.

\begin{theorem}\label{powers}
Suppose that $a=t^r$ and $b=t^s$ for some positive integer $t$ and non-negative integers $r,s$. Then
\begin{equation}\label{me2}
(a^n+1)(b^n+1)=x^2,\quad a<b,
\end{equation}
holds if and only if one of the following occurs:
\begin{itemize}
\item[(i)] $n=3$ and $(a,b,x)=(1,23,156)$;
\item[(ii)] $n=2$ and $(a,b,x)=(1,u_k,2v_k)$ for some integer $k\geq 2$, where $(u_k,v_k)$ denotes the $k$-th positive integer solution of the negative Pell equation $u^2-2v^2=-1$;
\item[(iii)] $n=1$ and $(a,b,x)=(1,2m^2-1,2m)$ for some integer $m\geq2$;
\item[(iv)] $n=1$ and $(a,b,x)=(7,49,20)$.
\end{itemize}
\end{theorem}

The next theorem concerns the case of even exponents.  

\begin{theorem}\label{t2}
Let $g=\gcd(a,b)$ and suppose that  
\[
(a^n+1)(b^n+1)=x^2,\qquad a<b,
\]
has a solution $(n,x) \in \N^2$ with $n$ even.
\begin{itemize}
\item[(i)] If $g=1$, then $n=2$ and
\[
(a,b,x)=(u_r,u_s,2v_rv_s),\quad r<s,\quad \gcd(2r-1,2s-1)=1.
\]
where $(u_k,v_k)$ denotes the $k$-th positive integer solution of the negative Pell equation $u^2-2v^2=-1$.
\item[(ii)] If $1<g<a$, then $4\nmid n$ and $g^2 < a$.
\item[(iii)] If $g=a$, then $4\nmid n$ and $a^2 < b$.
\end{itemize}
\end{theorem}

Note, that while part (i) of Theorem~\ref{t2} completely resolves the case with even $n$ and $\gcd(a,b)=1$, parts (ii) and (iii) only provide necessary (but not sufficient) conditions for a solution to exist. For example, if $a=6$ and $b=10$, then $g=2$, so indeed $g^2<6$. However, it can be seen by consideration modulo $3$, that $(6^n+1)(10^n+1)=x^2$ is not solvable in positive integers. \par

The paper is organized as follows. In Section~\ref{sec0}, we collect several auxiliary results on Diophantine equations that will be used throughout the proofs. In Section~\ref{sec00}, we present some basic facts concerning the negative Pell equation $u^2-dv^2=-1$. Section~\ref{sec1} is devoted to the proof of Theorem~\ref{powers}, while Theorem~\ref{t2} is proved in Section~\ref{sec2}. Finally, in Section~\ref{sec3}, we discuss several examples and compare our results with those known for Szalay's original equation~\eqref{sza}. We conclude by stating two conjectures.

%%%%%%%%%%%%%%%%%%%%%%%%%%%%%%%%%%%%%%%%%%%%%%%%%%%%%%%%
%%%%%%%%%%%%%%%%%%%%%%%%%%%%%%%%%%%%%%%%%%%%%%%%%%%%%%%%%
%%%%%%%%%%%%%%%%%%%%%%%%%%%%%%%%%%%%%%%%%%%%%%%%%%%%%%%%

\section{\textbf{Auxiliary results on Diophantine equations}}\label{sec0}

We begin by stating several auxiliary results on Diophantine equations that will be used in the proofs. The first one is Mihăilescu’s theorem \cite{mih02}, formerly known as Catalan's conjecture.

\begin{lemma}[Mihăilescu {\cite{mih02}}]\label{mih}
The equation 
\[ 
x^m-y^n=1
\] 
has a unique solution in positive integers $x,y,m,n>1$, namely 
\[ 3^2-2^3=1. \] 
\end{lemma}

\begin{lemma}[Ljunggren {\cite{lj43}}]\label{y2}
The equation
\[
\frac{x^n+1}{x+1}=y^2
\]
has no solutions in positive integers $x,n,y$ with $x>1$ and odd $n\geq3$.
\end{lemma}

\begin{lemma}[Waall {\cite{ww72}}]\label{rw}
The only solutions to 
\[
x^3+1=2y^2
\]
in positive integers are $(x,y)=(1,1)$ and $(x,y)=(23,78)$. \end{lemma}

\begin{lemma}[Bennett and Skinner {\cite{bs04}}]\label{bs}
If $n\geq4$ is an integer, then the equation
\[
x^n+y^n=2z^2
\]
has no solutions in coprime positive integers $x,y,z$ satisfying $x>y$.
\end{lemma}

\begin{lemma}[Ljunggren {\cite{lj42}}; see also Cohn {\cite{cohn97}}]\label{n4} 
For every fixed positive integer $d$, the equation \[ x^4-dy^2=-1 \] has at most one solution in positive integers $x,y$. 
\end{lemma}

As usual, $\gcd(a,b)$ denotes the greatest common divisor of positive integers $a$ and $b$. The following lemma is standard, but for completeness, we include its proof, since we were unable to find a direct reference.

\begin{lemma}\label{gcdplus}
Let $t,r,s$ be positive integers and put $g=\gcd(r,s)$. Then
\[
\gcd(t^r+1,t^s+1)=
\begin{cases}
t^g+1, & \text{if } \frac{r}{g}\text{ and }\frac{s}{g}\text{ are odd},\\[4pt]
1 \text{ or } 2, & \text{otherwise}.
\end{cases}
\]
\end{lemma}

\begin{proof}
Let $d=\gcd(t^r+1,t^s+1)$. If $d=1$, then there is nothing to prove. Therefore, assume that $d>1$. Since $t^r\equiv t^s \equiv -1 \pmod d$, we obtain $t^{2r}\equiv t^{2s}\equiv 1 \pmod d$. We also see that the order of $t$ modulo $d$ divides $2g=2\gcd(r,s)$, which implies that $t^{2g} \equiv 1 \pmod d$. Consequently,
\begin{equation}\label{eq1}
d\mid t^{2g}-1=(t^g-1)(t^g+1).
\end{equation}
\textbf{Case I}. First, suppose that $r/g$ and $s/g$ are both odd. Then it is well-known that
\[
t^r+1=(t^g)^{r/g}+1\quad\text{and}\quad t^s+1=(t^g)^{s/g}+1
\]
are both divisible by $t^g+1$. Therefore, $d=d'\cdot (t^g+1)$ for some positive integer $d'$. From \eqref{eq1} it follows that $t^g-1 \equiv 0 \pmod{d'}$. Then, on the one hand, we have
\[
t^r+1=(t^g)^{r/g}+1 \equiv 2 \pmod{d'} \quad \text{and} \quad  t^s+1=(t^g)^{s/g}+1\equiv 2.
\]
On the other hand, from the definition of $d$ it follows that 
\[t^r+1\equiv t^s+1 \equiv 0 \pmod{d'}.\] 
Thus, $2 \equiv 0 \pmod{d'}$, which forces $d' \in \{1,2\}$. However, if $d'=2$, then $2(t^g+1)$ divides $t^r+1$. Using expansion of $(x^n-1)/(x-1)$, it is immediate to check that
\[\frac{(t^g)^{r/g}+1}{t^g+1}\equiv 1 \pmod {2},\]
due to the fact that $r/g$ is odd. Therefore, $d'=1$, that is 
\[d=t^g+1,\]
as claimed.\par
\textbf{Case II.} Next suppose that at least one of $r/g$ and $s/g$ is even. Since $r/g$ and $s/g$ are coprime, exactly one of them is even. Without loss of generality, assume that $r/g$ is even and $s/g$ is odd. Then modulo $t^g+1$ we have
\begin{equation}\label{eq22}
t^r+1=(t^g)^{r/g}+1\equiv 2 \quad \text{and}\quad  
t^s+1=(t^g)^{s/g}+1\equiv 0 ,
\end{equation}
while modulo $t^g-1$ we get
\begin{equation}\label{eq33}
t^r+1=(t^g)^{r/g}+1\equiv 2 \quad \text{and}\quad   t^s+1\equiv (t^g)^{s/g}+1\equiv 2.
\end{equation}
Recall from \eqref{eq1}, that $d=\gcd(t^r+1,t^s+1)$ divides $(t^g-1)(t^g+1)$. It is also clear that $\gcd(t^g-1, t^g+1)=2$. Therefore,  \eqref{eq22} combined with \eqref{eq33} forces $d \in \{1,2\}$. If $t$ is even, we get $d=1$, and if $t$ is odd, $d=2$. This completes the proof of the lemma.
\end{proof}

%%%%%%%%%%%%%%%%%%%%%%%%%%%%%%%%%%%%%%%%%%%%%%%%%%%%%%%%%
%%%%%%%%%%%%%%%%%%%%%%%%%%%%%%%%%%%%%%%%%%%%%%%%%%%%%%%%
%%%%%%%%%%%%%%%%%%%%%%%%%%%%%%%%%%%%%%%%%%%%%%%%%%%%%%%%%%%

\section{\textbf{Preliminaries on the negative Pell equation}}\label{sec00}

The search for solutions to $(a^n+1)(b^n+1)=x^2$ for even $n$ often leads to a system of negative Pell equations, namely
\[
a^n-dy_1^2=-1
\quad\text{and}\quad
b^n-dy_2^2=-1,
\]
where $dy_1^2\cdot dy_2^2=x^2$. Accordingly, we collect several facts concerning such equations. Let $d$ be a positive integer which is not a perfect square and consider the negative Pell equation
\begin{equation}\label{negp}
u^2-dv^2=-1. 
\end{equation}
Assume that \eqref{negp} is solvable in positive integers. Let $(u_1,v_1)$ denote its minimal positive integer solution, that is,
\[
u_1+v_1\sqrt d
=
\min\{u+v\sqrt d:\ u,v\in\mathbb N,\ u^2-dv^2=-1\}.
\]
It is well known that all positive integer solutions $(u_j,v_j)$ of \eqref{negp} are given by
\begin{equation}\label{formula}
u_j+v_j\sqrt d
=
(u_1+v_1\sqrt d)^{2j-1},
\qquad j\geq 1.
\end{equation}

Note that, in contrast to the Pell equation $u^2-dv^2=1$, which always has infinitely many integer solutions, the negative Pell equation $u^2-dv^2=-1$ may have no integer solutions at all. For further details, one may consult Keith Conrad's expository paper on Pell equations \cite{conp}. The next two lemmas establish several divisibility properties in the same spirit as those for the classical Pell equation. We include the proofs, since we were unable to find direct references for these results.

\begin{lemma}\label{gcd2}
Suppose that the negative Pell equation
\[
u^2-dv^2=-1
\]
is solvable in positive integers. Consider any two positive integer solutions
$(u_r,v_r)$ and $(u_s,v_s)$. Let $\gcd(2r-1,2s-1)=2g-1$. Then
\[
\gcd(u_r,u_s)=u_g
\qquad\text{and}\qquad
\gcd(v_r,v_s)=v_g.
\]
\end{lemma}

\begin{proof}
Put
\[
\varepsilon=u_1+v_1\sqrt d.
\]
For every integer $k\geq 0$, define integers $A_k$ and $B_k$ by
\begin{equation}\label{pellg}
\varepsilon^k=A_k+B_k\sqrt d.
\end{equation}
Since $u_1^2-dv_1^2=-1$, we have 
\begin{equation}\label{pella}
A_k^2-dB_k^2=(-1)^k.
\end{equation}
In particular, $\gcd(A_k,B_k)=1$ for every $k\geq 0$. By \eqref{formula}, for every $j\geq 1$ we have
\[
A_{2j-1}=u_j
\qquad\text{and}\qquad
B_{2j-1}=v_j.
\]
We first prove the assertion for the second coordinates. Assume that $n$ and $m$ are odd positive integers with $n>m$. After applying \eqref{pellg} to $\varepsilon^n=\varepsilon^m\varepsilon^{n-m}$ and collecting similar terms we obtain
\[
B_n=A_mB_{n-m}+B_mA_{n-m}.
\]
Hence
\[
\gcd(B_m,B_n)=\gcd(B_m,A_mB_{n-m}).
\]
Since $\gcd(A_m,B_m)=1$, it follows that
\[
\gcd(B_m,B_n)=\gcd(B_m,B_{n-m}).
\]
Repeating this argument, we obtain, by the Euclidean algorithm,
\[
\gcd(B_m,B_n)=B_{\gcd(m,n)}.
\]
Taking $m=2r-1$ and $n=2s-1$, we obtain
\[
\gcd(v_r,v_s)=v_g,
\]
where $\gcd(2r-1,2s-1)=2g-1$.\par
It remains to prove the assertion for the first coordinates. As before, $n$ and $m$ are odd positive integers with $n>m$. We first record a simple expression for negative powers of $\varepsilon$. From \eqref{pellg} and \eqref{pella}
we derive 
\[\varepsilon^{-k} = \frac{A_k-B_k\sqrt d}{(-1)^k} = (-1)^k(A_k-B_k\sqrt d). 
\]
In particular, if $k$ is odd, then 
\begin{equation}\label{pellgn}
\varepsilon^{-k}=-(A_k-B_k\sqrt d). 
\end{equation}
After applying \eqref{pellg} to $\varepsilon^n=\varepsilon^m\varepsilon^{n-m}$ and collecting similar terms, we obtain
\[ A_n=A_mA_{n-m}+dB_mB_{n-m}. \] 
Since $m$ is odd, we have 
\[ \varepsilon^{-m}=-(A_m-B_m\sqrt d). \] 
Therefore, after applying \eqref{pellgn} to
$\varepsilon^{n-2m}=\varepsilon^{n-m}\varepsilon^{-m}$ and collecting similar terms we get 
\[ A_{n-2m}=-A_mA_{n-m}+dB_mB_{n-m}. \]
Subtracting the last two equalities gives 
\[ A_n-A_{n-2m}=2A_mA_{n-m}. \]
Hence \[ A_n\equiv A_{n-2m}\pmod {A_m}, \] 
and therefore
\[ \gcd(A_m,A_n)=\gcd(A_m,A_{n-2m}). \] 
Since $n-2m$ is odd, we have $A_{n-2m}=\pm A_{|n-2m|}$. Thus
\[ \gcd(A_m,A_n)=\gcd(A_m,A_{|n-2m|}). \] The operation $ (m,n)\mapsto (m,|n-2m|)$ preserves the greatest common divisor of the indices and decreases the larger index. Repeating this reduction, we eventually obtain 
\[ \gcd(A_m,A_n)=A_{\gcd(m,n)}. \] 
Taking $m=2r-1$ and $n=2s-1$, we obtain 
\[ \gcd(u_r,u_s)=u_g, \]
where $\gcd(2r-1,2s-1)=2g-1$. This completes the proof.
\end{proof}

\begin{lemma}\label{div}
Suppose that the negative Pell equation \[ u^2-dv^2=-1 \] is solvable in positive integers. Then for every $j\geq1$ we have
\[ u_1\mid u_j \qquad\text{and}\qquad v_1\mid v_j. \] 
\end{lemma}

\begin{proof} This is a corollary of Lemma~\ref{gcd2} with $r=1$ and $s=j$. \end{proof}

In both Theorem~\ref{powers} and Theorem~\ref{t2}, we will encounter the equation $x^2+1=2y^2$. Accordingly, we record all its positive integer solutions.

\begin{lemma}\label{pel2}
The equation
\[
x^2-2y^2=-1
\]
has infinitely many solutions in positive integers. Moreover, every positive integer solution is of the form $(x,y)=(u_j,v_j)$ for some integer $j\ge1$, where
\[
u_j=\frac{\alpha^{2j-1}+\beta^{2j-1}}{2},
\qquad
v_j=\frac{\alpha^{2j-1}-\beta^{2j-1}}{2\sqrt2},
\]
with
\[
\alpha=1+\sqrt2,\qquad
\beta=1-\sqrt2.
\]
\end{lemma}

\begin{proof}
Observe that $x^2-2y^2=-1$ has minimal solution
$(u_1,v_1)=(1,1)$. By \eqref{formula}, all positive integer solutions are given by 
\[ u_j+v_j\sqrt2=(1+\sqrt2)^{2j-1},\qquad j\geq 1. \] Putting $\alpha=1+\sqrt2$ and $\beta=1-\sqrt2$, we get 
\[ u_j+v_j\sqrt2=\alpha^{2j-1} \quad\text{and}\quad u_j-v_j\sqrt2=\beta^{2j-1}. \]
Adding and subtracting these two identities gives
\[ u_j=\frac{\alpha^{2j-1}+\beta^{2j-1}}{2} \quad\text{and}\quad v_j=\frac{\alpha^{2j-1}-\beta^{2j-1}}{2\sqrt2}. \]
This proves the lemma. 
\end{proof}

%%%%%%%%%%%%%%%%%%%%%%%%%%%%%%%%%%%%%%%%%%%%%%%%%%%%%%%%
%%%%%%%%%%%%%%%%%%%%%%%%%%%%%%%%%%%%%%%%%%%%%%%%%%%%%%%%
%%%%%%%%%%%%%%%%%%%%%%%%%%%%%%%%%%%%%%%%%%%%%%%%%%%%%%%%

\section{\textbf{Proof of Theorem~\ref{powers}}}\label{sec1}
\begin{proof}
Since $a<b$, it follows that $r<s$. First we deal with the case $r=0$, which leads to $a=t^r=1$. Therefore, the main equation 
\begin{equation}\label{mn}
(a^n+1)(b^n+1)=x^2
\end{equation}
becomes
\[2(b^n+1)=x^2.\]
Consequently, 
\begin{equation}\label{aa1}
b^n+1=2y^2
\end{equation}
for some $y\in\N$ and $b>1$. If $n\geq4$, then from Lemma~\ref{bs} follows that $\eqref{aa1}$ has no solutions. If $n=3$, Lemma~\ref{rw} implies that $(b,y)=(23,78)$. Thus, $(a,b,x)=(1,23,156)$ is a solution of \eqref{mn}. If $n=2$, then solutions of \eqref{aa1} correspond to the positive solutions of the negative Pell equation 
\begin{equation}\label{aa3}
u^2-2v^2=-1,
\end{equation}
which is completely described in Lemma~\ref{pel2}. It remains to observe that for any solution $(u_k, v_k)$ of \eqref{aa3} with $u_k>1$, the triple $(a,b,x)=(1,u_k,2v_k)$ is a solution of \eqref{mn}. Finally, if $n=1$, then for any $m \in \N$ it is clear that $(b,y)=(2m^2-1,m)$ is a solution to \eqref{aa1}. Consequently $(a,b,x)=(1, 2m^2-1,2m)$ is a solution to \eqref{mn} for any integer $m\geq2$. This completes the proof for the case $a=1$.\par
For the remainder of the proof we assume that $a>1$, which implies that $t>1$ and $s>1$. We have
\[a^n+1=t^{rn}+1\quad\text{and}\quad b^n+1=t^{sn}+1,\quad r<s.\]
Set
\[d=\gcd(t^{rn}+1, t^{sn}+1)\quad\text{and}\quad g=\gcd(rn,sn).\]
Then $rn=ge$ and $sn=gw$ for some positive integers $e$ and $w$ satisfying  $\operatorname{gcd}(e,w)=1$. Since $a\neq b$, it is also clear that $e\neq w$. \par
\textbf{Case I}. Assume first that both $e=rn/g$ and $w=sn/g$ are odd. Then Lemma~\ref{gcdplus} implies that $d=t^{g}+1$. Dividing both sides of $(a^n+1)(b^n+1)=x^2$ by $d$ leads to
\begin{equation*}
\frac{t^{ge}+1}{t^{g}+1}\cdot \frac{t^{gw}+1}{t^{g}+1}=\left(\frac{x}{t^{g}+1}\right)^2
\end{equation*}
with 
\[\gcd\left(\frac{t^{ge}+1}{t^{g}+1}, \frac{t^{gw}+1}{t^{g}+1}\right)=1.\]
Therefore, by setting $t^{g}=z$, we derive that for some positive integers $y_1$ and $y_2$, the following must hold
\[\frac{z^{e}+1}{z+1}=y_1^2\quad\text{and}\quad\frac{z^{w}+1}{z+1}=y_2^2.\]
Since $t>1$, we have $z>1$. Therefore, Lemma~\ref{y2} implies that $e=w=1$, which forces $r=s$, a contradiction. \par
\textbf{Case II}. Next assume that at least one of $e=rn/g$ and $w=sn/g$ is even. Then, Lemma~\ref{gcdplus} implies that $d=1$ or $d=2$. \par
If $d=1$, then 
\[t^{rn}+1=y_1^2\quad\text{and}\quad t^{sn}+1=y_2^2\]
for some distinct positive integers $y_1$ and $y_2$. Recall that $r<s$. If $rn>1$, this implies that the equation $u^m+1=v^2$ has two distinct solutions $(u,v,m) \in \N^3$ with $m>1$. However, this is impossible according to Lemma~\ref{mih}. If $rn=1$, then $r=n=1$. Then Lemma~\ref{mih} implies that the only solution in positive integers of $t^{s}+1=y_2^2$ is $(t,s,y_2)=(2,3,3)$. Therefore, $t=2$, which leads to $2^1+1=y_1^2$, a contradiction. \par
If $d=2$, then for some positive integer $y_1$ and $y_2$, we must have
\begin{equation}\label{comb}
t^{rn}+1=2y_1^2\quad\text{and}\quad t^{sn}+1=2y_2^2.
\end{equation}
However, according to Lemma~\ref{bs}, for $t>1$ this is impossible if at least one of $rn$ or $sn$ is greater than $3$. Thus, $1\leq rn < sn \leq 3$. \par
If $rn=1$, then $r=n=1$ and $s=2$ since $s$ must be even. From \eqref{comb} we have then that 
\[t=2y_1^2-1\quad\text{and}\quad (2y_1^2-1)^2+1=2y_2^2.\]
By rearranging the second equation we get
\[y_2^2=2y_1^4-2y_1^2+1.\]
The corresponding quartic curve
\[
Y^2=2X^4-2X^2+1
\]
has genus $1$. A standard computation of integral points on genus $1$ curves using MAGMA's \texttt{IntegralQuarticPoints} routine shows that its integral points are precisely
\[
(X,Y)=(0,\pm1),(\pm1,\pm1),(\pm2,\pm5).
\]
Since $y_1 <y_2$, we obtain that $(y_1,y_2)=(2,5)$. Thus 
\[t=2\cdot 2^2-1=7\quad \text{and} \quad x^2=(t+1)(t^2+1)=20^2.\]
Thus, $(a,b,x)=(7,7^2,20)$ is the solution of the original equation \eqref{mn}.\par
If $rn=2$, then $sn=3$. This is possible only if $(n,r,s)=(1,2,3)$. Thus \eqref{comb} becomes
\[t^{2}+1=2y_1^2\quad\text{and}\quad t^{3}+1=2y_2^2.\]
By Lemma~\ref{rw}, for $t>1$, the equation $t^{3}+1=2y_2^2$ has the unique solution $(t,y_2)=(23,78)$. However, $23^2+1=2y_1^2$ has no solutions in positive integers. This completes the proof of Theorem~\ref{powers}. 
\end{proof}

%%%%%%%%%%%%%%%%%%%%%%%%%%%%%%%%%%%%%%%%%%%%%%%%%%%%%%%%%%%%
%%%%%%%%%%%%%%%%%%%%%%%%%%%%%%%%%%%%%%%%%%%%%%%%%%%%%%%%%%%%
%%%%%%%%%%%%%%%%%%%%%%%%%%%%%%%%%%%%%%%%%%%%%%%%%%%%%%%%%%%%

\section{\textbf{Proof of Theorem~\ref{t2}}}\label{sec2}

\begin{proof}
Write $n=2k$. Put 
\[d=\gcd(a^n+1,b^n+1).\]
Then $a^n+1=da_1$ and $b^n+1=db_1$ for some coprime positive integers $a_1$ and $b_1$. From
\[x^2=(a^n+1)(b^n+1)=d^2a_1b_1\]
we see that actually both $a_1$ and $b_1$ must also be squares. Hence, $a_1=y_a^2$ and $b_1=y_b^2$ for some coprime positive integers $y_a$ and $y_b$.  It is also clear that $d$ is not a square, for if $d=t^2$, then $(ty_a)^2-(a^k)^2=1$, which is impossible. Hence
\begin{equation}\label{s2}
a^{2k}+1=dy_a^2\quad\text{and}\quad b^{2k}+1=dy_b^2.
\end{equation}
Therefore
\begin{equation}\label{sols}
(a^k,y_a)=(u_r,v_r)\quad\text{and}\quad (b^k,y_b)=(u_s,v_s),\quad r<s,
\end{equation}
where $(u_r, v_r)$ and $(u_s, v_s)$ denote two distinct solutions of the negative Pell equation 
\begin{equation}\label{ppl}
u^2-dv^2=-1.
\end{equation}
Let $(u_1,v_1)$ be the minimal solution of \eqref{ppl}. The remaining proof is split in three cases according to the value of $g=\gcd(a,b)$. \par
Case (i). Assume that $g=1$. By Lemma~\ref{div}, $u_1\mid a^k$ and $u_1\mid b^k$. Therefore, $u_1=1$. 
Analogously, $v_1\mid y_a$ and $v_1\mid y_b$. Since $\gcd(y_a,y_b)=1$, we get $v_1=1$. Consequently, $(u_1,v_1)=(1,1)$, and therefore from \eqref{ppl} we get $d=2$. Accordingly, \eqref{s2} becomes
\begin{equation}\label{ev1}
a^{2k}+1=2y_a^2\quad\text{and}\quad b^{2k}+1=2y_b^2.
\end{equation}
If $2k\geq4$, then Lemma~\ref{bs} implies that the equation $X^{2k}+Y^{2k}=2Z^2$ does not have solution in coprime positive integers with $X>Y$. Thus, \eqref{ev1} cannot hold if $b>a\geq1$.
On the other hand, if $2k=2$, then $(a,y_a)$ and $(b,y_b)$ are two distinct solutions of the negative Pell equation $x^2-2y^2=-1$, whose positive integer solutions are completely described in Lemma~\ref{pel2}. It remains to observe that in this case  
\[(a^2+1)(b^2+1)=(u_r^2+1)(u_s^2+1)=2v_s^2\cdot 2v_r^2=(2v_sv_r)^2.\]
Thus, if $n=2$, then $a=u_r$, $b=u_s$ and $x=2v_rv_s$. Also, since $a<b$ and $\gcd(a,b)=1$, we require $r<s$ and
$\gcd(u_r,u_s)=1$. By Lemma~\ref{gcd2}, the latter condition is equivalent to $\gcd(2r-1,2s-1)=1$. This completes the proof of the case $g=1$. \par 
Cases (ii) and (iii). Assume that $g>1$. Since $a<b$, it is also clear that $g\leq a$. Set $\gcd(2r-1,2s-1)=2e-1$. Then by Lemma~\ref{gcd2}, $\gcd(u_r,u_s)=u_e$ and $\gcd(v_r,v_s)=v_e$. Since $v_r=y_a$ and $v_s=y_b$ are coprime, it follows that $v_e=1$. Therefore, $e=1$. This implies that 
\[u_1=\gcd(u_r,u_s)=\gcd(a^k, b^k)=\gcd(a,b)^k=g^k.\]
Thus, the minimal solution of $u^2-dv^2=-1$ is $(u_1,v_1)=(g^k,1)$. This leads to $g^{2k}-d=-1$, or equivalently $g^{2k}+1=d$. Write $a=gw$. Since $d\mid a^n+1$, we have $g^n+1\mid g^n w^n+1$. Observe that $g^n\equiv -1\pmod{g^n+1}$, so $g^n w^n+1\equiv 1-w^n\pmod{g^n+1}$. Therefore $g^n+1\mid w^n-1$. \par
If $w>1$ (which corresponds to (ii), $1<g<a$), then $g^n+1 \leq w^n-1$, which forces $g < w$. Thus, $g^2<gw=a$. If $w=1$ (which corresponds to (iii), $g=a$), then $b=aw'$ for some $w'>1$. Then by exactly the same reasoning we deduce that $a^n+1\mid (w')^n-1$, which implies $a^2<aw'=b$. \par
It remains to show that $4 \nmid n$. Suppose on the contrary, that $4 \mid n$. Write $n=4\ell$. Then \eqref{s2} becomes 
\[(a^{\ell})^4-dy_a^2=-1\quad\text{and}\quad (b^{\ell})^4-dy_b^2=-1.\]
This means that for some fixed $d$, the equation
\begin{equation}\label{nd444}
X^4-dY^2=-1
\end{equation}
has at least two solutions in positive integers. However, this is a contradiction to Lemma~\ref{n4}, which states that for any $d$ at most one solution can occur to \eqref{nd444}. The proof of the theorem is complete.
\end{proof}

%%%%%%%%%%%%%%%%%%%%%%%%%%%%%%%%%%%%%%%%%%%%%%%%%%%%%%%%%%%%
%%%%%%%%%%%%%%%%%%%%%%%%%%%%%%%%%%%%%%%%%%%%%%%%%%%%%%%%%%%%%%%%%%%%%%%%%%%%%%%%%%%%%%%%%%%%%%%%%%%%%%%%%%%%%%%%%%%%%%%%

\section{\textbf{Some Examples and Conjectures}}\label{sec3}

In this section, we compare the results known for the original equation of Szalay
\begin{equation}\label{sz}
(a^n-1)(b^n-1)=x^2,\quad a<b,
\end{equation}
with those obtained for the equation considered in this paper, namely
\begin{equation}\label{vr}
(a^n+1)(b^n+1)=x^2,\quad a<b.
\end{equation}
One immediate difference is that equation \eqref{vr} is defined for $a=1$, whereas \eqref{sz} is not. Another contrast appears when $a$ and $b$ are distinct powers of the same positive integer. In this setting, Cohn \cite{cohn02} proved that  \eqref{sz} is solvable in four cases: 
\[(a,b)\in\{(2,2^2), (3,3^5), (7, 7^4), (c^2-1,(c^2-1)^2)\},\quad c\geq2.\] 
For the plus equation, the corresponding phenomenon is much more restrictive. Apart from the solutions with $a=1$ listed in Theorem~\ref{powers}, \eqref{vr} is solvable only in the case
\[
(a,b)=(7,7^2).
\]
A similar contrast occurs for exponents divisible by $4$. The original equation of Szalay \eqref{sz} has a unique solution with $4\mid n$, which occurs for $(a,b)=(13,239)$ and $n=4$. \cite{cohn02}. In the plus case, by Theorem~\ref{t2}, equation \eqref{vr} has no solutions whenever $4\mid n$. Apart from this distinction, the statement of Theorem~\ref{t2} closely parallels Theorems~2 and 5 of Keskin \cite{kes19} for equation \eqref{sz}. In both cases, the analysis of even exponents reduces to a system of two Pell equations. Next, we provide several examples illustrating how Theorem~\ref{t2} can be applied.

\begin{corollary}\label{cr}
Let
\[
(a,b)\in \{(2,3),(7,41),(8,12),(7,28)\}.
\]
Then \eqref{vr} has a solution in positive integers $(n,x)$ only for $(a,b)=(7,41)$, in which case
\[
(n,x)=(2,290).
\]
\end{corollary}

\begin{proof}
If $(2^n+1)(3^n+1)=x^2$ holds for some $(n,x)\in\N^2$, then consideration modulo $5$ implies that $n$ is even. Since $\gcd(2,3)=1$, Theorem~\ref{t2} yields $n=2$. However, $(2^2+1)(3^2+1)=50$ is not a square, so the equation has no solutions.
Similarly, if $(7^n+1)(41^n+1)=x^2$ holds for some $(n,x)\in\N^2$, then consideration modulo $13$ implies that $n$ is even. Indeed, it is easy to verify that for $a\equiv7\pmod{13}$ and $b\equiv2\pmod{13}$, the congruence
$(a^n+1)(b^n+1)\equiv x^2 \pmod{13}$ is solvable only when $n$ is even. Since $\gcd(7,41)=1$, Theorem~\ref{t2} again gives $n=2$. Finally, $(7^2+1)(41^2+1)=290^2$, and therefore $(n,x)=(2,290)$ is the unique solution.\par
If either $(8^n+1)(12^n+1)=x^2$ or $(7^n+1)(28^n+1)=x^2$
holds for some $(n,x)\in\N^2$, then consideration modulo $5$ implies that $n$ is even. For the pair $(8,12)$, we have $g=\gcd(8,12)=4$. Since $g<8$ and $g^2=16>8$, part~(ii) of Theorem~\ref{t2} implies that the equation has no solutions. For the pair $(7,28)$, we have $g=\gcd(7,28)=7$. Since $7^2 > 28$, part~(iii) of Theorem~\ref{t2} implies that the equation has no solutions.
\end{proof}
One of the main open problems concerning equation \eqref{sz} is Cohn's conjecture that no solutions exist for $n\geq5$ \cite{cohn02}. Another conjecture of Cohn asserts that the only solutions with $n=3$ are 
\[(a,b,x)\in\{(2,4,21),(2,22,273),(3,313,28236),(4,22,819)\}.\]
It can be shown that this conjecture would follow from proving that, for square-free $d$, the equation \[X^3-dY^2=1\]
has two or more solutions $(X,Y) \in \N^2$ only if $d \in \{7, 26\}$. These conjectures of Cohn naturally lead to analogous questions for the plus equation \eqref{vr}. Motivated by the results obtained in this paper, we propose the following.
\begin{conjecture}
The Diophantine equation 
\[(a^n+1)(b^n+1)=x^2,\quad a<b,\]
has no solutions in positive integers for $n\ge 4$.
\end{conjecture}
Furthermore, a computer search for $n=3$ and $1\leq a<b\leq 10^6$ motivates the following conjecture for the cubic case.
\begin{conjecture}
The Diophantine equation
\[
(a^3+1)(b^3+1)=x^2,\qquad a<b,
\]
has only the following positive integer solutions:
\[
(a,b,x)\in\{(1,23,156),(6,26,1953),(20,362,616077)\}.
\]
\end{conjecture}
We note that the cubic conjecture would follow from proving that, for square-free $d$, the equation 
\[X^3-dY^2=-1\]
has two or more solutions $(X,Y) \in \N^2$ only if $d \in \{2,217,889\}$.

\renewcommand{\bibname}{Bibliography}
\bibliographystyle{acm}
\bibliography{References}

\end{document}